\documentclass[english,11pt]{amsart}
\usepackage[english]{babel}
\usepackage{amsmath,amssymb,enumerate,amsthm}
\usepackage[latin1]{inputenc}
\usepackage[T1]{fontenc}
\usepackage{cite}
\usepackage{latexsym}
\usepackage{setspace}

\newtheorem{lemma}{Lemma}[section]
\newtheorem{theorem}[lemma]{Theorem}

\newtheorem{coro}[lemma]{Corollary}

\newtheorem{rk}[lemma]{Remark}

\newtheorem*{theorem*}{Theorem}
\newtheorem*{lemma*}{Lemma}

\def\N{\mathbb{N}}
\def\C{\mathbb{C}}
\def\R{\mathbb{R}}
\def\Z{\mathbb{Z}}
\def\Q{\mathbb{Q}}

\def\Bad{\mathrm{Bad}}
\def\H{\mathcal{H}}
\def\W{\mathcal{W}}

\def\={\;=\;}
\def\.={\;\dot{=}\;}

\begin{document}

\title[Metric number theory of Fourier coefficients of Modular Forms]{Metric number theory of Fourier coefficients of Modular Forms}
\author{Paloma Bengoechea}
\address{ETH Zurich \\ Department of Mathematics\\ Ramistrasse 101\\8092 Zurich\\Switzerland}
\email{paloma.bengoechea@math.ethz.ch}
\thanks{Bengoechea's research is supported by SNF grant 173976.}

\begin{abstract} 
We discuss the approximation of real numbers by Fourier coefficients of newforms, following recent work of  Alkan, Ford and Zaharescu. The main tools used here, besides the (now proved) Sato-Tate Conjecture, come from 
metric number theory.
\end{abstract}

\maketitle

\section{Introduction}

Consider a newform 
$$
f(z)=\sum_{n=1}^\infty a_f(n) e^{2\pi i nz}
$$ 
 of  even integer weight $k$ for $\Gamma_0(N)$. Suppose that $f(z)$ is normalized so that $a_f(1)=1$. We denote by $S^{new}_k(N)$ the set of such functions. In \cite{AFZ} Alkan, Ford and Zaharescu approximate a given real number $x$ by the 
 normalized Fourier  coefficients 
\begin{equation}\label{Fnc}
a(n)=\dfrac{a_f(n)}{n^{(k-1)/2}},
\end{equation}
assuming that the $a_f(n)$ are integer and
 satisfy a divergence hypothesis.
 
It is well-known that the  coefficients \eqref{Fnc}
 satisfy
\begin{align}
&a(mn)=a(m)a(n)\quad\mbox{if $(m,n)=1$},\nonumber\\
&a(p^m)=a(p)a(p^{m-1})-a(p^{m-2})\quad\mbox{for $p$ prime},\label{recursion}\\
&|a(n)|\leq d(n),\nonumber
\end{align}
where $d(n)$ is the divisor function.
The inequality
was proven by Deligne as a consequence of his proof of the Weil Conjectures.
The results above determine everything about $a(n)$ except for the distribution
of the $ a(p)\in [-2,2]$. 
For any prime $p$, we define the angle $0\leq \theta_p\leq \pi$ such that
\begin{equation}\label{angle}
a(p)=2\cos \theta_p.
\end{equation}

The coefficients $a(p)/2$ are equidistributed in
  the interval $[-1,1]$ over primes $p$ with respect to different measures according to whether $f$ has complex multiplication (CM) or not.  
The Sato-Tate Conjecture, now a theorem due to Barnet-Lamb, Geraghty, Harris,
and Taylor \cite{BL}, asserts that, if $f$ has non-CM, then the $a(p)$ are equidistributed  with respect to the Sato-Tate measure
\begin{equation}\label{mST}
\dfrac{2}{\pi} \sqrt{1-t^2} \, dt.
\end{equation}
This means that
for any fixed interval $[\alpha,\beta]\subseteq [-1,1]$,
\begin{equation}\label{intST}
\# \left\{p\leq x:\, a(p)\in [\alpha,\beta]\right\}\sim \left(\dfrac{2}{\pi}\int_\alpha^\beta \sqrt{1-t^2} \, dt\right)\pi(x),
\end{equation}
where $\pi(x)$ as usual denotes the number of primes $p\leq x$.
For bounds on the error term in \eqref{intST} see \cite{BuK}, \cite{RT}.
 
 If $f$ has CM, it is a classical result that the $a(p)/2$ are equidistributed in $[-1,1]$ with respect to the measure 
\begin{equation} \label{mCM}
 \dfrac{1}{2}\delta_0([\alpha,\beta]) + \dfrac{1}{2\pi} \dfrac{1}{\sqrt{1-t^2}}\, dt,
 \end{equation}
 where $\delta_0$ denotes the Dirac measure at zero.
 This follows from the work of Deuring on the  equidistribution of the values of Hecke characters (see \cite{Hecke} and \cite{fite} for a nice 
 expository note). 
 




With the aim of a better understanding of the distribution of the coefficients \eqref{Fnc}, several studies on the distribution of primes $p$ for which $a(p)=c$ for some fixed $c\in\R$  and the distribution of integers $n$ for which $a(n)\neq 0$ have been carried (see for example \cite{S}, \cite{MMS}, \cite{El}, \cite{FM} \cite{RT}, etc.). In this note, we propose a different and new approach that relates to Diophantine approximation and metric number theory. We study the approximation of any real number  by the coefficients \eqref{Fnc} with techniques and results commonly used and developed in the area of metric number theory, such as Schmidt's Game (see  section 4 for the definition and properties of this game). We  hope that these will contribute to give a better understanding of the distribution of the $a(n)$. 
We also hope that the reader may find some interest in how we relate the two areas of modular forms and Diophantine approximation here.

The study of the approximation of any given  real number by the coefficients $a(n)$ has been 
introduced by 
 Alkan, Ford and Zaharescu in \cite{AFZ} with  the  theorem below. This
result was motivated and inspired by metric number theory connections exploited
in \cite{AGZ} and \cite{AFZ0}. 

\begin{theorem*}[AFZ]
Let 
$$
f(z)=\sum_{n=1}^\infty a_f(n)e^{2\pi inz}
$$
be a newform  in $S^{new}_k(N)$. Assume that the coefficients $a_f(n)$ are integer and satisfy 
\begin{equation}\label{condition}
\sum_{p} a(p)^2=\infty,
\end{equation}
where the summation is over all primes. Then for any real number $x$, there exists a positive constant $C_{f,x}$ depending only on $f$ and $x$, such that
\begin{equation}\label{thineq}
|a(n)-x|\leq\dfrac{C_{f,x}}{\log n}
\end{equation}
holds for infinitely many positive integers $n$.
\end{theorem*}

Thanks to Sato-Tate, one can remove now the integrality and divergence hypothesis \eqref{condition} on the Fourier coefficients.
The condition that the $a(n)$ are integers is quite restrictive  in the case of the full modular group (when $N=1$), since it is expected that there are no 
Hecke newforms of weight $k>22$ with integer coefficients (as a consequence of  Maeda's Conjecture, see \cite{HM}). 
However, Alkan, Ford and Zaharescu presented an
example of an infinite set of newforms of weight $2$ arising from elliptic curves in their paper where the AFZ Theorem applies.
Also  the 
inequality \eqref{thineq} holds in fact for infinitely many integers of the form $n=p^m$, where $p\geq 5$ is a fixed prime.  Theorem \ref{th refinement} comprises these refinements.
\\

The statements of Theorems \ref{th refinement} and (AFZ) hold for all $x$, and in this sense they can be viewed as analogs of Dirichlet's theorem in the classical theory of Diophantine approximation. 
In \cite{AFZ} the question of developing a metric number theory is asked. For example, what would the rate of approximation be for almost all $x$? We study this problem in section 3.

One may also ask what is the best rate of approximation for all $x$ in (AFZ) and Theorem \ref{th refinement}.
In this direction, we show in Theorem \ref{thW1}, that for any prime $p$, there are $x$ for which 
 there exists a constant $\gamma_{x,f,p}$ such that
\begin{equation}\label{ineqw}
|a(p^m)-x|>\dfrac{\gamma_{x,f,p}}{m^2}\quad\forall m\in\N.
\end{equation}
Moreover, we show that, for every prime $p$, the set of $x\in (0,\frac{1}{2\sin\theta_p})$ satisfying \eqref{ineqw} 
has full Hausdorff dimension. 
In fact we prove that this set  is winning with respect to the  famous Schmidt Game.

Throughout we denote by $\|\cdot\|$ the distance to a nearest integer.

\section{Refinement of Theorem (AFZ)}
 
 Throughout this section we fix a newform $f$ of level $N$. 
We follow the ideas of \cite{AFZ} but we  use the equidistribution results on the angles $\theta_p$ to get rid of the assumptions in Theorem (AFZ).  
 
 It follows from Deligne's proof 
of the Ramanujan conjecture for newforms (c.f. \cite{D}) that, for any prime number $p$,
$$
a(p)=\mu_p+\overline{\mu_p}
$$
for some $\mu_p\in\C$ with $|\mu_p|=1$. Using the recursion \eqref{recursion}, it follows by induction on $m$ that 
\begin{equation}\label{regular}
a(p^m)=\dfrac{\mu_p^{m+1}-\overline{\mu_p}^{m+1}}{\mu_p-\overline{\mu_p}} =\dfrac{\sin((m+1)\theta_p)}{\sin\theta_p}.
\end{equation}
This relation plays a key role in \cite{AFZ} and here again. 

We need information about the distribution of the numbers $\theta_p/2\pi$ that are
  irrational. This is provided by the following lemma.

 \begin{lemma}\label{lemma1}
The angles $\theta_p/2\pi$ that are irrational are dense in $[0,\frac{1}{2}]$. 
 \end{lemma}
 
 \begin{proof} 
  
Since the angles $\theta_p$ are dense in $[0,\pi]$, it is enough to show that the $\theta_p/2\pi$ that are rational constitute a finite set.
Let $p$ be a prime such that  $\theta_p/2\pi$ is rational. We write $\theta_p/2\pi=A_p/B_p$ where $B_p\geq 1$, $A_p< B_p$ and $(A_p,B_p)=1$. Since $\mu_p=e^{2\pi i(A_p/B_p)}$ is a primitive $B_p$-th root of unity, we have $[\Q(\mu_p):\Q]=\phi(B_p)$, where $\phi$ is the Euler function. On the other hand, $\mu_p$ is a zero of the polynomial
\begin{align*}
P(z)&=(z-\mu_p)(z-\overline{\mu_p})(z+\mu_p)(z+\overline{\mu_p})\\
&=(z^2-a(p)z+1)(z^2+a(p)z+1)=z^4+(2-a(p)^2)z^2+1.
\end{align*}
It is well-known that the field $\mathbb{K}$ generated by the coefficients of the newforms in $S^{new}_k(N)$ is a finite extension over $\Q$; let $d$ be the degree. Since $P(z)$ has coefficients in $\mathbb{K}$, it follows that $[\Q(\mu_p):\Q]\leq 4d$. Hence $\phi(B_p)\leq 4d$, and therefore $B_p$ belongs to a finite set of integers independent of $p$, and so does $\theta_p$. 
Hence the lemma follows.

\end{proof}






\begin{theorem}\label{th refinement}
For any real number $x$, there exist infinitely many primes $p$ and a constant $C_{f,x}$ depending only on $f$ and $x$ such that
$$
|a(p^m)-x|\leq \dfrac{C_{f,x}}{m}
$$
for infinitely many positive integers $m$.
\end{theorem}

In the proof we use the following well-known theorem due to Minkowski (see for example \cite{Hua}).

\begin{theorem*}[M]\label{th1}
For any irrational number $\theta$ and any real number $x$, 
$$
\left\|m\theta + x\right\| < \dfrac{3}{m}
$$
 for infinitely many positive integers $m$.
\end{theorem*}

\begin{proof}[Proof of Theorem \ref{th refinement}]

Let $x\in\R$. Let $p$ be a prime satisfying that $|\sin\theta_p|<1/|x|$ and that $\theta_p/2\pi$ is irrational. By Lemma \ref{lemma1} there exist infinitely many such primes. 
 Let
  $\delta$ be the angle $0\leq \delta<2\pi$ such that $\sin \delta=x \sin\theta_p$. 

By Theorem (M)  there exist infinitely many positive integers $m$ such that
$$
\left\|(m+1)\dfrac{\theta_p}{2\pi} - \dfrac{\delta}{2\pi}\right\| = 
\left|(m+1)\dfrac{\theta_p}{2\pi}-\left[(m+1)\dfrac{\theta_p}{2\pi}-\dfrac{\delta}{2\pi}\right]-\dfrac{\delta}{2\pi}\right|
< \dfrac{3}{m+1},
$$
where $[\cdot]$ is either the floor or the ceiling part.
For each $m$ as above, we have
$$
|\sin((m+1)\theta_p)-\sin\delta| < \dfrac{6\pi}{m+1}.
$$
Hence, by \eqref{regular},
$$
|a(p^m)-x|<\dfrac{C_{f,x}}{m+1}
$$
with $C_{f,x}=6\pi/\sin \theta_p$.

\end{proof}

\begin{rk}
In Theorem \ref{th refinement} one cannot improve the `for infinitely many primes $p$' statement to a `for all primes $p$' statement. Indeed, when $f$ is associated to an elliptic curve $E/\Q$, in which case $k=2$, by the modularity theorem, the primes $p\geq 5$ such that $p\nmid N$ and $a(p)=0$ are primes for which $E$ has supersingular reduction. Elkies \cite{El} proved that infinitely many such primes exist. 
For such a prime $p$, it follows from the recursion \eqref{recursion} that $a(p^m)=0$ for all odd $m$ and $a(p^m)=\pm a(p^{m+2})$ for all even $m$.
Hence the sequence $a(p^m)$ is not even dense.
\end{rk} 

\section{Metrical theory with respect to Lebesgue measure}

Throughout, we denote by $m$ the Lebesgue measure.
 Given a real number $x$ and its positive continued fraction $[a_0,a_1,a_2,\ldots]$ where $a_0\in\Z$ and $a_i\in\N$ for all $i\geq 1$, we refer to the denominators of the rational approximations $[a_0,a_1,\ldots,a_n]$ as the \textit{convergents} $q_n$ of $x$.

\begin{theorem}\label{thL}

Let $\varphi:\N\rightarrow(0,+\infty)$ be a  decreasing function.  
  The inequality
$$
|a(n)-x|<\varphi(n) 
$$
has infinitely many solutions $n>0$ for almost no $x$ if $\sum_{n=1}^\infty\varphi(n)<\infty$.

If the angles $\theta_p/2\pi$ that are irrational and satisfy
\begin{equation}\label{cond a}
\sum_{r=0}^\infty \sum_{n=q^{(p)}_r}^{q^{(p)}_{r+1}-1} \min\left(\varphi(n),\|q^{(p)}_n\theta_p\|\right)=\infty,
\end{equation}
where $q^{(p)}_{n}$ are the convergents of $\theta_p/2\pi$, are dense in $[0,\frac{1}{2}]$, then there exists a constant $C_{f,x}$  depending only on $f$ and $x$ such that
the inequality 
$$
|a(n)-x|<C_{f,x}\, \varphi(n)
$$
has infinitely many solutions $n>0$  for almost all $x$.


\end{theorem}

The proof of the theorem strongly relies on the following recent result in metrical  inhomogeneous Diophantine approximation due to Fuchs and Kim \cite{FK}:

\begin{theorem*}[FK]
Let $\varphi:\N\rightarrow(0,+\infty)$ be a  decreasing function and $\theta$ be
an irrational number with  convergents $s_r/q_r$. Then, for almost all
$x\in\R$,
\begin{equation}\label{yep}
\|m\theta - x\|<\varphi(m)\quad \mbox{for infinitely many $m\in\N$}
\end{equation}
if and only if
$$
\sum_{r=0}^\infty \sum_{m=q_r}^{q_{r+1}-1} \min\left(\varphi(m),\|q_m\theta\|\right)=\infty.
$$
\end{theorem*}

\begin{proof}[Proof of Theorem \ref{thL}]
The first part of the theorem is a consequence of the Borel-Cantelli Lemma. Indeed, let $\W$ denote the set of all real numbers $x$ such that
$$
|a(n)-x|<\varphi(n)
$$
for infinitely many $n\in\N$.
Define
$$
E_n=\left\{x\in\R:\, |a(n)-x|<\varphi(n)\right\}.
$$
We can rewrite $\W$ as
$$
\W=\bigcap_{M=1}^\infty \bigcup_{n=M}^\infty E_n.
$$
Hence, for every integer $M>1$,
$$
\W\subseteq \bigcup_{n=M}^\infty E_n,
$$
and then
$$
m(\W)\leq \sum_{n=M}^\infty m(E_n)=2\sum_{n=M}^\infty\varphi(n)\rightarrow 0\quad \mbox{as $M\rightarrow\infty$}
$$
since $\sum_{n=1}^\infty\varphi(n)<\infty$.
\\

The second part is obtained as an application of Theorem (FK). We first prove that the inequality
$$
|a(p^m)-x| < C_{f,x} \varphi(m)
$$
has infinitely many solutions $m>0$ for almost all $x$ in the interval $(-\frac{1}{\sin \theta_p},\frac{1}{\sin \theta_p})$ for any fixed prime $p$ such that $\theta_p/2\pi$ is irrational and satisfies \eqref{cond a}. Since we assume that the set of such primes is dense, this will prove the second part of the theorem.
Let $p$ be a prime  such that $\theta_p/2\pi$ is irrational. We suppose that $\theta_p/2\pi$ satisfies the condition \eqref{cond a}.
By Theorem (FK), for almost all $\delta\in\R$,
\begin{equation}\label{aux}
\left\|(m+1)\dfrac{\theta_p}{2\pi} - \delta\right\|< \varphi(m)\quad\mbox{for infinitely many $m\in\N$}.
\end{equation}
 Let $\delta\in\R$ satisfy \eqref{aux}. 
For each $m$ as above, we have
$$
|\sin((m+1)\theta_p)-\sin 2\pi\delta| < 2\pi \varphi(m).
$$
Hence for $x=\sin \delta/\sin\theta_p$ we have 
$$
|a(p^m)-x|<C_{f,x}\, \varphi(m)
$$
with $C_{f,x}=2\pi/\sin \theta_p$. The set of such $x$ has maximal Lebesgue measure in the interval $[-1/\sin\theta_p,1/\sin\theta_p]$.

\end{proof}

\begin{rk}
 The result of Fuchs and Kim contains several previous results as special cases: the theorem of Kurzweil \cite{Ku} as well as its extensions given in the same paper,
 and the theorem of Tseng \cite{T2}. Hence one could change the condition \eqref{cond a} by the respective conditions of Kurzweil and Tseng in Theorem \eqref{thL}. In particular, if $\theta_p$ is badly approximable, then by definition there exists a constant $c>0$ such that 
 $\|n\theta_p\|>c/n$ for all $n\geq 1$. Thus, for $q^{(p)}_r\leq n\leq q^{(p)}_{r+1}$, 
 $$
 \|q^{(p)}_r\theta_p\|>\dfrac{c}{n},
 $$ 
 so the condition \eqref{cond a} is satisfied provided that $\sum_{n\geq 1} \min(\varphi(n),c/n)=\infty$, hence provided that  $\sum_{n=1}^\infty \varphi(n)=\infty$. (This is Kurzweil's condition). 
\end{rk} 
 
\begin{coro}\label{coroL} Let $\varphi:\N\rightarrow(0,+\infty)$ be a  decreasing function. If the  angles $\theta_p/2\pi$ that  are badly approximable are dense in $[0,\frac{1}{2}]$, then 
there exists a constant $C_{f,x}$ depending only on $f$ and $x$ such that  the inequality 
$$
|a(n)-x|<C_{f,x}\, \varphi(n)
$$
has infinitely many solutions $n>0$  for almost all $x$ if $\sum_{n=1}^\infty\varphi(n)=\infty$.
\end{coro}

\section{On the optimality of Theorem \ref{th refinement}}

Let $X\subset\mathbb{R}$. For $s \geq 0$ and $\rho > 0$ define
$$
\H^s_\rho(X)=\inf\left\{\sum_i d_i^s:\left\{B_i\right\} \mbox{is a $\rho$-cover of $X$}\right\},
$$
where a $\rho$-cover of $X$ is any countable collection $B_i$ of balls of diameter $d_i<\rho$ such that
$$
X\subset \bigcup_i B_i,
$$
and the infimium is taken over all possible $\rho$-covers of $X$. The $s$-dimensional Hausdorff measure of $X$ is defined as the following (finite or
infinite) limit
$$
\H^s(X)=\lim_{\rho\rightarrow 0^+}\H^s_\rho(X),
$$
and the Hausdorff dimension of $X$ is 
$$
\dim_H (X) = \inf\left\{s \geq 0 : \H^s (X) = 0\right\}.
$$

\subsection{Schmidt Game}

We present a simplified version of the game that Wolfgang M. Schmidt introduced in \cite{Schmidt}. Let $I\subseteq\R$ be an interval. The game involve two players A and B and two real numbers $\alpha,\beta\in (0,1)$. The game starts with player B choosing at will a closed interval $B_0\subseteq I$. Next, player A chooses a closed interval $A_0\subset B_0$ of length $\alpha|B_0|$. Then, player B chooses  a closed interval $B_1\subset A_0$ of length $\alpha\beta|B_0|$. The two players keep playing alternately in this way, generating a nested sequence of closed intervals in $I$:
$$
B_0\supset A_0\supset B_1\supset A_1\supset\ldots\supset B_s\supset A_s\supset\ldots
$$
with lengths
$$
|A_s|=\alpha|B_s|\qquad\mbox{and}\qquad|B_s|=\beta |A_{s-1}|=(\alpha\beta)^s|B_0|.
$$
A subset $X\subseteq I$ is called $\alpha$-winning in $I$ if player A can play so that the unique point of intersection
$$
\bigcap_{s=0}^\infty B_s=\bigcap_{s=0}^\infty A_s
$$
lies in $X$ whatever the value of $\beta$ is. A set is simply called winning in $I$ if it is $\alpha$-winning in $I$ for some $\alpha\in(0,1)$. When $I=\R$, we simply say that $X$ is winning or $\alpha$-winning.
Schmidt proved the following two theorems about his game:

\begin{theorem*}[S1]
If $X\subseteq I$ is  winning in $I$, then $\dim_H(X)=\dim_H(I)$.
\end{theorem*}

\begin{theorem*}[S2]
The intersection of countably many $\alpha$-winning sets in $I$ is $\alpha$-winning in $I$.
\end{theorem*}

Schmidt also proved that the winning property is invariant under local isometries; this has been refined by Dani (see \cite{Dani} Proposition 5.3) with the following result:

\begin{theorem*}[D]
The image of a winning set in $I$  under a bi-Lipschitz map $\varphi$ is winning in $\varphi(I)$. 
\end{theorem*}

\subsection{The set $\Bad(p,\varphi)$}
Let $p$ be a prime and $\varphi:\N\rightarrow(0,+\infty)$ be a  decreasing function.
In view of Theorems \ref{th refinement} and \ref{thL}, we define the set of $(p,\varphi)$ badly approximable numbers by
$$
\Bad(p,\varphi)=\left\{x\in \R:\, \inf_{m\in\N} \varphi(m)\cdot|a(p^m)-x|>0\right\}.
$$

\begin{theorem}\label{thW1} 
For any prime $p$ the set $\Bad(p,m^2)$
is winning in $(0,\frac{1}{2\sin \theta_p})$.
\end{theorem}

In order to prove Theorem \ref{thW1}, we introduce the set of numbers usually called \textit{twisted badly approximable numbers}: 
$$
\Bad^\vee_\theta:=\left\{x\in\R:\, \inf_{m\in\N} m\cdot \|m\theta-x\|>0\right\}.
$$

The following result is due to Tseng \cite{Tseng}:

\begin{theorem*}[T]
For any $\theta\in\R$, the set $\Bad^\vee_\theta$
is $1/8$-winning.
\end{theorem*}

\begin{proof}[Proof of Theorem \ref{thW1}]
Let $p$ be a prime. Let us consider the sets
$$
T^{1/2}(-\Bad^\vee_{\theta_p/2\pi}):=\left\{x\in\R:\, \dfrac{1}{2}-x\in\Bad^\vee_{\theta_p/2\pi}\right\},
$$
$$
T^{-1/2}(-\Bad^\vee_{\theta_p/2\pi}):=\left\{x\in\R:\, -\dfrac{1}{2}-x\in\Bad^\vee_{\theta_p/2\pi}\right\}.
$$
 Define 
$$
X:=\Big(0,\dfrac{1}{8}\Big)\cap
\Bad^\vee_{\theta_p/2\pi}\cap T^{1/2}(-\Bad^\vee_{\theta_p/2\pi}) \cap T^{-1/2}(-\Bad^\vee_{\theta_p/2\pi}).
$$
By Theorems (D) and (S2), $X$
is winning in $(0,\frac{1}{4})$. Then Theorem  (D) again implies that $\frac{\sin 2\pi X}{\sin \theta_p}$  is winning in $(0,\frac{1}{2\sin \theta_p})$.
If we show that 
\begin{equation}\label{inclusion}
\dfrac{\sin (2\pi X)}{\sin\theta_p}\subseteq \Bad(p,m^2),
\end{equation} 
 then $\Bad(p,m^2)$ will also be winning in $(0,\frac{1}{2\sin \theta_p})$. Indeed, from the definition of Schmidt's Game, it is clear that  a winning  strategy for player A  consists in choosing respective sets $A_0,\ldots,A_s,\ldots$ that give a winning strategy for the subset in the  left hand side of  \eqref{inclusion}.
  
Next we show \eqref{inclusion}.
Let $\delta\in X$. By definition, 
there exists a positive constant $\gamma$
depending on $\delta$ and $\theta_p$ such that the following inequalities hold for all positive integers $m$:
\begin{align}
&\left|(m+1)\dfrac{\theta_p}{2\pi}-\left[(m+1)\dfrac{\theta_p}{2\pi}\right]-\delta\right|
\geq \left\|(m+1)\dfrac{\theta_p}{2\pi}-\delta\right\|
>\dfrac{\gamma}{m+1}, \label{al:1}\\
&\left|(m+1)\dfrac{\theta_p}{2\pi}-\left[(m+1)\dfrac{\theta_p}{2\pi}\right]-\dfrac{1}{2}+\delta\right|
\geq \left\|(m+1)\dfrac{\theta_p}{2\pi}-\dfrac{1}{2}+\delta\right\|
>\dfrac{\gamma}{m+1}, \label{al:2}\\
&\left|(m+1)\dfrac{\theta_p}{2\pi}-\left[(m+1)\dfrac{\theta_p}{2\pi}\right]+\dfrac{1}{2}+\delta\right|
\geq \left\|(m+1)\dfrac{\theta_p}{2\pi}+\dfrac{1}{2}+\delta\right\|
>\dfrac{\gamma}{m+1}, \label{al:3}
\end{align}
 where $[x]$ denotes a nearest integer to $x$.
 For convenience later, we choose 
 \begin{equation}\label{gamma}
 \gamma<\frac{1}{4}.
 \end{equation}
Set
$$
\alpha=(m+1)\dfrac{\theta_p}{2\pi}-\left[(m+1)\dfrac{\theta_p}{2\pi}\right].
$$ 
Clearly $-\frac{1}{2}\leq\alpha\leq\frac{1}{2}$ and
$$
\sin2\pi\alpha=\sin((m+1)\theta_p).
$$
By \eqref{al:1}, we have that
$$
\dfrac{\gamma\pi}{m+1}
<\left|\pi(\alpha-\delta)\right|\leq\dfrac{\pi}{2}+\dfrac{\pi}{8},
$$
so (using also \eqref{gamma})
\begin{equation}\label{buenosin}
\left|\sin \pi(\alpha-\delta)\right|>\left|\sin \left(\dfrac{\gamma\pi}{m+1}\right)\right|.
\end{equation}
On another hand, by \eqref{al:2} either
$$
\dfrac{\gamma\pi}{m+1}+\dfrac{\pi}{2}<\pi(\alpha+\delta)<\dfrac{\pi}{2}+\dfrac{\pi}{8},
$$
or
$$
\pi(\alpha+\delta)<\dfrac{\pi}{2}-\dfrac{\gamma\pi}{m+1}.
$$
By \eqref{al:3} we  have also that either
$$
\pi(\alpha+\delta)>-\dfrac{\pi}{2}+\dfrac{\gamma\pi}{m+1},
$$
or
$$
-\dfrac{\pi}{2}-\dfrac{\pi}{8}<\pi(\alpha+\delta)<-\dfrac{\gamma\pi}{m+1}-\dfrac{\pi}{2}.
$$
Therefore (using also \eqref{gamma})
\begin{equation}\label{buenocos}
\left|\cos \pi(\alpha+\delta)\right|>
\left|\cos \left(\dfrac{\pi}{2}-\dfrac{\gamma\pi}{m+1}\right)\right|=\left|\sin\left(\dfrac{\gamma\pi}{m+1}\right)\right|.
\end{equation}

The inequalities \eqref{buenosin} and \eqref{buenocos} imply that
\begin{align*}
|\sin ((m+1)\theta_p) - \sin 2\pi\delta|&=				2\left|\sin \pi(\alpha-\delta) \cos\pi(\alpha+\delta)\right|\\
&>2\sin \left(\dfrac{\gamma\pi}{m+1}\right)^2\\
&>\dfrac{2\gamma^2\pi^2}{(m+1)^2} \left(1- \dfrac{\gamma^2\pi^2}{6(m+1)^2}\right)^2.
\end{align*}

Hence, by using \eqref{regular} and \eqref{gamma}, there exists a positive constant $c$ depending only on $\delta$, $p$ and $f$ such that 
$$ 
\left|a(p^m) - \dfrac{\sin 2\pi\delta}{\sin \theta_p}\right| > \dfrac{c}{(m+1)^2}.
$$
This proves \eqref{inclusion}.
\end{proof}


\section{Further questions/remarks}

Theorem \ref{thW1} seems to indicate that Theorem \ref{th refinement} is optimal, i.e. that $\log$ is the best possible rate of approximation on a sequence $\left\{n=p^m\right\}_{m\in\N}$ for $p$ a fixed prime. But we have  no evidence that supports that $\log$ remains the best rate for the sequence $\left\{n\in\N\right\}$. From a metrical point of view, it follows from Corollary \ref{coroL} that $\log$ is not the best rate if the angles $\theta_p/2\pi$ that are badly approximable are dense in $[0,\frac{1}{2}]$. In this case, the rate of approximation can be made any power of $\log$.
Hence an interesting question would  be  to  investigate the distribution of the $\theta_p$ that  are badly approximable. In fact, do all the $\theta_p/2\pi$ that are irrational have the same rate of approximation by the rationals?
\\

Changing perspective slightly, one could fix $p$ and try to approximate a real number $x$ by the Fourier coefficients $a_f(p)$ as $f(z)$ varies. In this context,
Conrey, Duke and Farmer and simultaneously Serre obtained equidistribution results for the $\theta_p$ or equivalently the normalized $a(p)$.
The first three authors show in \cite{CDF} via the Selberg trace formula that, as $k\to\infty$, the set $\left\{\theta_f(p):\, f\in S_k(1)\right\}$ becomes uniformly distributed with respect to the measure
\begin{equation}\label{measure Pl}
\dfrac{2}{\pi}\left(1+\dfrac{1}{p}\right)\dfrac{\sin^2\theta}{\left(1+\frac{1}{p}\right)^2+\frac{4}{p}\sin^2\theta} d\theta.
\end{equation}
The above measure is the $p$-adic Plancherel measure, and it
is also the spectral measure of the nearest-neighbor Laplacian on a
$p+1$ regular tree (see \cite{LPS}). 
 Serre \cite{S} showed that, as $f(z)$ varies on $S_k(N)$ for any sequence $(k,N)\rightarrow\infty$, the normalized $a(p)$ are equidistributed with respect to the measure corresponding to \eqref{measure Pl} in $[-1,1]$.

\end{document}